\documentclass{amsart}

\usepackage{amsfonts,amssymb,latexsym,amsmath,graphicx,bbm, pinlabel}
\input xy
\xyoption{all}

\usepackage[margin=1in,marginparwidth=0.8in, marginparsep=0.1in]{geometry}

\usepackage[bookmarks=true, bookmarksopen=true,%
bookmarksdepth=3,bookmarksopenlevel=2,%
colorlinks=true,%
linkcolor=blue,%
citecolor=blue,%
filecolor=blue,%
menucolor=blue,%
urlcolor=blue]{hyperref}

\usepackage[all, knot, poly]{xy}
\xyoption{knot}

\usepackage{times}
\usepackage{mathrsfs}

\newtheorem{lemma}{Lemma}
\newtheorem{proposition}[lemma]{Proposition}
\newtheorem{corollary}[lemma]{Corollary}
\newtheorem{theorem}[lemma]{Theorem}
\newtheorem{conjecture}[lemma]{Conjecture}

\theoremstyle{definition}

\theoremstyle{remark}
\newtheorem*{remark}{Remark}

\newcommand{\F}{\mathbb{F}}

\newcommand{\R}{\mathbb{R}}

\newcommand{\Z}{\mathbb{Z}}

\newcommand{\cA}{\mathcal{A}}
\newcommand{\cB}{\mathcal{B}}

\newcommand{\cR}{\mathcal{R}}

\newcommand{\e}{\epsilon}
\newcommand{\coeffs}{\mathbbm{k}}
\newcommand{\Aug}{{\cA ug}_+}
\newcommand{\AugBC}{{\cA ug}_-}
\newcommand{\alg}{\cA}
\newcommand{\dga}{dga}
\newcommand{\dgas}{dgas}

\DeclareMathOperator{\Hom}{Hom}

\DeclareMathOperator{\Aut}{Aut}
\DeclareMathOperator{\Span}{Span}
\DeclareMathOperator{\Coeff}{Coeff}

\begin{document}

\title{The cardinality of the augmentation category of a Legendrian link}

\author{Lenhard Ng}
\address{Lenhard Ng \\ Department of Mathematics \\ Duke University}
\email{ng@math.duke.edu}

\author{Dan Rutherford}
\address{Dan Rutherford \\ Department of Mathematical Sciences \\ Ball State University}
\email{rutherford@bsu.edu}

\author{Vivek Shende}
\address{Vivek Shende \\ Department of Mathematics \\ UC Berkeley}
\email{vivek@math.berkeley.edu}

\author{Steven Sivek}
\address{Steven Sivek \\ Department of Mathematics \\ Princeton University}
\email{ssivek@math.princeton.edu}

\begin{abstract}
We introduce a notion of cardinality for the augmentation category associated to a Legendrian knot or link in standard contact $\mathbb{R}^3$. This `homotopy cardinality' is an invariant of the category and allows for a weighted count of augmentations, which we prove to be determined by the ruling polynomial of the link. We present an application to the augmentation category of doubly Lagrangian slice knots.
\end{abstract}

\maketitle

\section{Introduction}

To a Legendrian knot or link $\Lambda$ in the standard contact $\R^3$,
one can associate a semi-free noncommutative \dga{} $\cA(\Lambda)$, the
Chekanov--Eliashberg \dga{}, whose generators
are the Reeb chords of $\Lambda$, and whose differential counts certain
holomorphic disks in the symplectization \cite{Eli, Che, EGH}.  Legendrian isotopies
$\Lambda \sim \Lambda'$ induce homotopy equivalences
$\cA(\Lambda) \sim \cA(\Lambda')$.  In practice, information is often extracted from $\cA(\Lambda)$
by studying its \dga{} maps to a field, called augmentations.  For instance, one could just count them:
that is, fix a finite field $\F_q$
and count \dga{} morphisms $\cA(\Lambda) \to \F_q$.

A series of works \cite{CP, Fuc, FI, Sab, NS, HR} gives increasingly precise relationships between
this count and certain combinatorial objects called {\em graded normal rulings}.
A {\em normal ruling} of a Legendrian front is a decomposition into
boundaries of embedded disks, each containing one right cusp and one left cusp, and satisfying
certain additional requirements; see Figure \ref{fig:m821} for an example and any of the above references for a more precise definition. The data of a normal ruling is just the specification, at each crossing, of whether the two disk boundaries simply trace out the knot, or `switch' at the crossing.
A normal ruling is {\em graded}
if all switches occur between strands of equal Maslov potential. The conditions defining graded normal rulings are such that a generating family (or sheaf) for the Legendrian link determines such a ruling by applying Barannikov normal form to the filtered complexes over each $x$ coordinate \cite{Bar,CP}.

The switching data can be viewed as
prescribing a gluing of the disks into a surface, and the Euler characteristic of this surface is a
natural invariant of the ruling.  For a given Legendrian front, we write $\cR$ for the set of graded normal rulings,
and $\chi(R)$ for the Euler characteristic of the surface associated to a ruling $R \in \cR$; note that $\chi(R) = \# \mbox{right cusps} - \# \mbox{switches}$.  Following \cite{CP}, we can then
assemble these into the {\em ruling polynomial} $R_\Lambda(z) = \sum_{R\in\cR} z^{-\chi(R)}$.

To state the relationship between rulings and the count of augmentations, we define the {\em shifted Euler characteristic}
\[
\chi_*(\Lambda) = \sum_{i \ge 0} (-1)^i r_i + \sum_{i < 0} (-1)^{i+1} r_i
\]
as in \cite{NS},
where $r_i$ is the number of Reeb chords whose corresponding element in $\cA(\Lambda)$ has
degree $i$.  If $\Lambda$ has $\ell$ components, we now have the
following result of Henry and Rutherford \cite[Remark~3.3(ii)]{HR}:
\begin{equation} \label{eq:hr}
 q^{-\chi_*(\Lambda) / 2} (q-1)^{-\ell}  \cdot  \# \{ \cA(\Lambda) \to \F_q \} =
R_\Lambda(q^{1/2}-q^{-1/2}).
\end{equation}

There is a certain tension in the above equation.
The right hand side of \eqref{eq:hr} has been proven by Chekanov and
Pushkar \cite{CP} to be an invariant of Legendrian isotopy, by
checking Legendrian Reidemeister moves.
Moreover, a result of Rutherford \cite{Rut} asserts that in the $2$-periodic setting, the right hand side is in fact a topological invariant, a part of the HOMFLY-PT polynomial of the underlying topological knot.

But on the left hand side, since
only the homotopy class of the algebra $\cA(\Lambda)$ is an invariant of the Legendrian
isotopy class of $\Lambda$, the set of augmentations is {\em not} an invariant.  However,
there is now
a unital $A_\infty$ category $\Aug(\Lambda; \F_q)$,
the {\em augmentation category}, whose objects are augmentations, and
whose endomorphisms are a version of linearized contact cohomology \cite{NRSSZ}.  This category is a
Legendrian isotopy invariant of $\Lambda$, in the sense that Legendrian isotopies induce
$A_\infty$-equivalences of categories.  It was introduced in order to match
a certain category of sheaves studied in \cite{STZ}---the subcategory
of compactly supported,
microlocal rank one objects in the category
$Sh_{\Lambda}(\R^2; \F_q)$ of
sheaves on the front plane with microsupport on $\Lambda$.
This category of sheaves is in turn equivalent to the full subcategory
consisting of objects asymptotic to $\Lambda$ in
the derived infinitesimally wrapped Fukaya category \cite{NZ}.  These categories
are also known to be equivalent to a category of generating families \cite{She}.

Our goal here is to show that the left hand side of Equation~\eqref{eq:hr}
can be replaced by the cardinality, suitably interpreted, of the
augmentation category $\Aug(\Lambda; \F_q)$.

Already for usual categories, there are two ways to count objects.\footnote{More properly,
we are counting in the groupoid of isomorphisms in the category.  } The first is to
just count the number of isomorphism classes.  The second, usually more natural, is to compute the
sum over the isomorphism classes of
the inverse of the number of automorphisms of an object in the class.
The reason to prefer the latter to the former is familiar from the theory of finite group
actions on sets.  Given a set $S$ carrying the action
of a group $G$, one can form the quotient set $S/G$, whose elements are equivalence
classes, or one can form the groupoid quotient $[S/G]$: its objects
are the $G$-orbits, and their endomorphisms are given
by the stabilizers.  One has $\#(S / G) = \# S / \# G$ only when the
$G$ action is free, while by
the orbit-stabilizer theorem, one always has $\# [S/G] = \#S / \#G$.

However, computer experiments
have shown that
neither counting equivalence classes nor computing the groupoid cardinality
of
$H^0(\Aug(\Lambda; \F_q))$ will give the left hand side of
Equation~\eqref{eq:hr}, even up to multiplication by powers of $q^{1/2}$ and $q-1$.  For
the count of equivalence classes, the Legendrian $m(8_{21})$ knot studied in \cite{MS} (and pictured in
Figure \ref{fig:m821}) has $16$ augmentations over $\F_2$ which fall into $10$ equivalence classes, whereas the left side of Equation~\eqref{eq:hr} is $4\sqrt{2}$.
For the groupoid cardinality, consider the Legendrian $m(9_{45})$ knot discussed in \cite[Section~4.4.4]{NRSSZ}.  This has $5$ augmentations over $\F_2$, none of which are equivalent; three of them have $\lvert\Aut(\e)\rvert=1$ and two have $\lvert\Aut(\e)\rvert=2$, so the groupoid cardinality is $4$, whereas the left side of Equation~\eqref{eq:hr} is equal to $5/\sqrt{2}$.

\begin{figure}

\centerline{ \includegraphics[scale=.8]{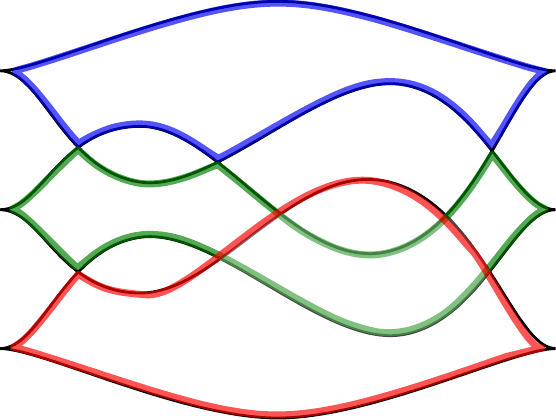}  }
\caption{  A graded normal ruling $R$ of a Legendrian $m(8_{21})$ knot with $\chi(R) = -1$.
}
\label{fig:m821}
\end{figure}

What does work is the homotopy cardinality \cite{BD}, also sometimes called the multiplicative
Euler characteristic.  This is a number attached to spaces with finitely many homotopy groups,
all of finite cardinality, and is defined as:

$$\# X := \sum_{[x] \in \pi_0(X)} \frac{|\pi_2(x, X)| |\pi_4(x, X)| \cdots}{|\pi_1(x, X)| |\pi_3(x, X)|
\cdots}.$$

The connection to the previously introduced notion of groupoid
cardinality comes from identifying $[S/G]$
with the homotopy-theoretic quotient---i.e., replacing $S$ with a space homotopy equivalent
to it on which $G$ acts freely, and then taking the quotient---whereupon
$\pi_0([S/G])$ is identified with the set of orbits,  $\pi_1(s, [S/G])$ is identified with the stabilizer
of $s$, and all higher $\pi_i$ are trivial.  One previous use of the
notion of homotopy cardinality was to make
sense of the Dijkgraaf--Witten theory with a higher groupoid as gauge group
\cite{JP}.

To apply this notion here, we take the groupoid part of the augmentation category.
That is, we write $\pi_{\ge 0} \Aug(\Lambda; \F_q)^*$ for
the groupoid consisting of isomorphisms in the truncation
$\pi_{\ge 0} \Aug(\Lambda; \F_q)$.  We have cohomological conventions for maps
in $\Aug$, so $\pi_{\ge 0}$ means that we take only morphisms in non-positive degree.
The homotopy cardinality of (the classifying space of) this groupoid
is then given by the formula:
\[ \# \pi_{\ge 0} \Aug(\Lambda; \F_q)^* = \sum_{[\epsilon] \in \Aug(\Lambda; \F_q)/\sim}
\frac{1}{\lvert\Aut(\epsilon)\rvert} \cdot \frac{| H^{-1} \Hom(\epsilon, \epsilon) | \cdot | H^{-3} \Hom(\epsilon, \epsilon) |
\cdots}{|H^{-2} \Hom(\epsilon,
\epsilon)| \cdot |H^{-4} \Hom(\epsilon,\epsilon)| \cdots}, \]
where $\Aut(\epsilon)$ is the set of units of $H^*\Hom(\epsilon,\epsilon)$.  Here we relate this quantity to the naive count of augmentations:

\begin{theorem} \label{thm:card}
We have:
$$\# \pi_{\ge 0} \Aug(\Lambda; \F_q)^*
= q^{\frac{tb - \chi_*}{2}} \cdot (q-1)^{-\ell}
 \cdot \#  \{ \cA(\Lambda) \to \F_q \}.$$
\end{theorem}

\begin{remark} This suggests that $\# \pi_{\ge 0} \Aug(\Lambda; \F_q)^*$ is
a quotient of the set $\{ \cA(\Lambda) \to \F_q \}$ by some higher groupoid
whose cardinality is $q^{\frac{\chi_* - tb}{2}} \cdot (q-1)^{\ell}$.
\end{remark}

We deduce the following result about counting objects in the sheaf category,
conjectured in \cite[Conjecture~7.5]{STZ}.

\begin{corollary}  \label{cor:ruling-polynomial}
We have the equalities:
$$\# \pi_{\ge 0} Sh_{\Lambda; 1} (\R^2; \F_q)^* =
\# \pi_{\ge 0} \Aug(\Lambda; \F_q)^*
= q^{tb(\Lambda)/2} R_\Lambda(q^{1/2}-q^{-1/2}).$$
\end{corollary}
\begin{proof}
The first equality holds by \cite{NRSSZ}, since both sides only depend on the cohomology
categories, which are isomorphic.  The second equality follows by combining Theorem
\ref{thm:card} with Equation~\eqref{eq:hr}.
\end{proof}

One application of this theorem is a strong restriction on Legendrian knots which we call \emph{doubly Lagrangian slice}.  A theorem of Eliashberg and Polterovich \cite{EP} (cf.\ \cite{Cha-symmetric}) implies that any Lagrangian concordance from the standard Legendrian unknot $U$ with $tb(U)=-1$ to itself, i.e., a Lagrangian cylinder in $\R \times \R^3$ which agrees with the product $\R \times U$ outside a compact set, can in fact be identified with $\R \times U$ by a compactly supported Hamiltonian isotopy.  It is possible, however, that for some nontrivial Legendrian $\Lambda$ there are Lagrangian concordances from $U$ to $\Lambda$ and from $\Lambda$ to $U$---in other words, that a Lagrangian concordance $L$ from $U$ to itself might satisfy
\[ L \cap \big((-\epsilon,\epsilon) \times \R^3\big) = (-\epsilon,\epsilon) \times \Lambda \]
for some $\epsilon > 0$---and in this case we say that $\Lambda$ is doubly Lagrangian slice.  In Section~\ref{sec:doubly-slice} we prove the following.
\begin{theorem}
If $\Lambda$ is doubly Lagrangian slice, then there is an $A_\infty$ quasi-equivalence $\Aug(\Lambda;\F_q) \cong \Aug(U;\F_q)$ over any finite field $\F_q$.
\end{theorem}
\noindent Given an augmentation $\e :\thinspace \cA(\Lambda) \to \F_q$, one can complete $\cA(\Lambda)$ to the $I$-adic completion $\widehat{\cA}(\Lambda)$, where $I$ is the ideal $\ker\e \subset \cA(\Lambda)$.
Using the above, we also prove that there is a quasi-isomorphism $\widehat{\cA}(\Lambda) \to \widehat{\cA}(U)$, where $\widehat{\cA}(U)$ is easily computed to be the power series ring $\F_q[[a]]$ with trivial differential, whenever $\Lambda$ is doubly Lagrangian slice.  See Proposition~\ref{prop:completion} for details.

So far, our discussion has been for augmentations and rulings with a full $\Z$ grading. One can also consider rulings and the augmentation category graded in $\Z/(2m)$ for some $m\in\Z$, where as mentioned earlier the $m=1$ case is related to the HOMFLY-PT polynomial. The complexes in $\Aug(\Lambda;\F_q)$ are now $2m$-periodic. It is then unclear what
should substitute for the homotopy cardinality $\# \pi_{\ge 0}
\Aug(\Lambda; \F_q)^*$, but we provide a candidate in
Corollary~\ref{cor:2mgraded} that we conjecture is related to the
$2m$-graded ruling polynomial of $\Lambda$ in the same way as
Theorem~\ref{thm:card}. For the $2$-periodic case, for example, we expect the homotopy cardinality to be given by
\[
\sum_{[\epsilon] \in \Aug(\Lambda; \F_q)/\sim}
\frac{1}{\lvert\Aut(\epsilon)\rvert} |H^{\operatorname{ev}}\Hom(\e,\e)|^{1/2} q^{-\ell/2},
\]
where $H^{\operatorname{ev}}$ denotes cohomology in even degree.

In Section~\ref{sec:counting}, we prove our main result, Theorem~\ref{thm:card}.
We discuss the application to doubly Lagrangian slice knots in
Section~\ref{sec:doubly-slice}, and the $\Z/(2m)$-graded setting in Section~\ref{sec:2m-graded}.

\subsection*{Acknowledgments}
We would like to thank the American Institute of Mathematics for
sponsoring a SQuaRE group, ``Sheaf theory and Legendrian knots'', at whose 2015 meeting we began the work for this paper. We are also grateful to the other participants of the SQuaRE, David Treumann, Harold Williams, and Eric Zaslow, for helpful conversations.  We also thank Andr\'e Henriques, Brad Henry, and Qiaochu Yuan.
The work of LN was supported by NSF grant DMS-1406371 and a grant from the Simons Foundation (\#341289 to Lenhard Ng).
The work of VS was supported by NSF grant DMS-1406871 and a Sloan fellowship.
The work of SS was supported by NSF grant DMS-1506157.

\section{Counting}
\label{sec:counting}

We turn to the proof of Theorem \ref{thm:card}.  Throughout this section, we consider augmentations with values in a field $\coeffs$, which will be assumed to be a finite field $\F_q$ when counting is necessary.
The basic idea to prove Theorem~\ref{thm:card} is to establish
a correspondence between the (not necessarily closed) degree zero automorphisms
of a given augmentation on the one hand, and isomorphic augmentations on the other.
The correspondence will be many to one, but the discrepancy will be accounted for by
homotopies, higher homotopies, etc.

The starting point is the comparison between morphisms in the augmentation category
and \dga{} homotopies.  Recall that if $f, f': \cA \to \cB$ are \dga{} morphisms, then
a chain homotopy between them, i.e., a map $h: \cA \to \cB$ with $d_{\cB} h \pm h d_{\cA} =
f - f'$, is a \dga{} homotopy between them if $h$ is an $(f, f')$-derivation, i.e.,
$h(ab) = f(a) h(b) \pm h(a) f'(b)$.

Recall also from \cite{NRSSZ} that in the case where $\Lambda$ is a knot,
if Hom spaces of the augmentation
category are calculated by perturbing the $2$-copy via a Morse function on the knot
with a single pair of critical points (the ``Lagrangian projection perturbation''), then
each $\Hom(\e,\e')$ is generated as a module by certain terms $a_i^+$ corresponding to
the Reeb chords of the original knot, together with a degree zero element $y^+$ coming from
one of the critical points of the Morse function, and a degree one element $x^+$ coming from
the other.  The element $-y^+$ is a strict identity on self-hom spaces, which we write as
$e_\e \in \Hom(\e,\e)$.  Note that the degree of the generator of the Hom space corresponding
to a Reeb chord is $1$ greater than the degree of the corresponding generator of the \dga{}.
There is a precise relation between \dga{} homotopies and closed isomorphisms in $\Aug(\Lambda)$:

\begin{proposition}[{\cite[Proposition~5.16]{NRSSZ}}] \label{prop:nrssz-homotopy}
Every general element of $\Hom^0(\e,\e')$ of the form
\[ \alpha := -y^+ - \sum k_i a^+_i \]
is invertible, meaning that there are elements
$\beta,\gamma\in\Hom^0(\e',\e)$ such that $m_2(\beta,\alpha) = e_{\e}$
and $m_2(\alpha,\gamma) = e_{\e'}$.  Moreover, we have $m_1(\alpha)=0$
if and only if the $(\e,\e')$-derivation $K: \cA(\Lambda) \to \coeffs$
defined by requiring $K(a_i)=k_i$ for each $i$ is a \dga{} homotopy from $\e$ to $\e'$.
\end{proposition}

If $\Lambda$ is a link with $\ell$ components, labeled $1,2,\dots,\ell$, we define a function $r \times c$ on the set of Reeb chords such that $r(a_i)$ and $c(a_i)$ are the labels of the overcrossing and undercrossing strands of $\Lambda$ at the chord $a_i$.  (This is called a \emph{link grading} in \cite{NRSSZ}.)  We assume that $\Lambda$ has a single base point on each component, and the generators $t_i^{\pm 1}$ of $\cA(\Lambda)$ correspond to the base point on component $i$; we will write $r(t_i^{\pm 1}) = c(t_i^{\pm 1}) = i$.  In this case each $\Hom(\e,\e')$ will have generators $y_1^+,\dots,y_\ell^+$ in degree zero and $x_1^+,\dots,x_\ell^+$ in degree one, corresponding to the minima and maxima of a Morse function on $\Lambda$ by which we perturb each component, in addition to the Reeb chord generators $a_i^+$.

In the case where $\Lambda$ has multiple components, the characterization of invertible elements and isomorphisms in $\Aug(\Lambda)$ closely resembles the one described in Proposition~\ref{prop:nrssz-homotopy}.  We will explain this below, showing in Proposition~\ref{prop:unique-isomorphism} that the (not necessarily closed) automorphisms $\alpha \in \Hom^0(\e,\e)$, including those for which $m_1(\alpha) \neq 0$, can instead be viewed as (closed) isomorphisms from $\e$ to some uniquely determined $\e'$.

\begin{lemma} \label{lem:construct-K}
For fixed $\e$, and arbitrarily chosen $k_i \in \coeffs$, there is a unique augmentation $\e': \cA \to \coeffs$ and a unique $(\e, \e')$-derivation $K$ such that $\e'(t_i)=\e(t_i)$ for $i=1,\dots,\ell$; $K(a_i) = k_i$ for all $a_i$; and $\e-\e' = K \circ \partial$.
\end{lemma}

\begin{proof}
We make use of the height filtration on the \dga{}, originally due to Chekanov \cite{Che}, which is defined using the lengths of the Reeb chords: if  $a_i$ has height $h(a_i)$, and the differential counts a disk $u$ with positive end at $a_0$ and negative ends at $b_1,b_2,\dots,b_m$, then $h(a_0) > \sum h(b_i)$, since by Stokes's theorem their difference is the area of the projection $\pi_{xy}(u)$.  Thus if the Reeb chords $a_1,\dots,a_r$ are ordered by increasing height, we have $\partial a_i \in F_{i-1} \cA := \coeffs\langle a_1, \ldots, a_{i-1}\rangle$.  We will define $\e'$ inductively by height.

Suppose that $\e'$ has been defined on $F_{i-1} \cA$ so that on $F_{i-1}\cA$ we have $\e' \circ \partial =0$ and $\e-\e' = K \circ \partial$, where $K$ is the extension of the map $a_i \mapsto k_i$ to $F_{i-1} \cA$ as an $(\e,\e')$-derivation.  Then we set  $\e'(a_i) :=  \e(a_i) - K(\partial a_i)$ and extend $\e'$ to $F_i \cA$ as an algebra homomorphism, and we can extend $K$ to $F_i\cA$ as an $(\e,\e')$-derivation satisfying $K(a_i) = k_i$.  We need to check that (i) $\e': (F_i\cA,\partial) \to (\coeffs,0)$ is a chain map, and (ii) $\e-\e' = K\circ\partial$ continues to hold on $F_{i}\cA$.  Note that the identities $\e'\circ \partial_\cA = \partial_\coeffs \circ \e' = 0$ and $\e-\e' = \partial_\coeffs \circ K + K \circ \partial_\cA = K \circ \partial_\cA$ will be satisfied by the algebra maps $\e,\e':\cA \to \coeffs$ and the $(\e,\e')$-derivation $K$ if they are satisfied on all of the generators of $\cA$.

To verify (i), since $\partial a_i \in F_{i-1} \cA$, the inductive hypothesis gives
\[
\e'(\partial a_i) = \e(\partial a_i) - K \circ \partial(\partial a_i) = 0.
\]
For (ii), the identity holds by construction when applied to $a_i$, so it holds on all of $F_i\mathcal{A}$.
\end{proof}

\begin{lemma} \label{lem:rescale-epsilon}
If $\e: \cA(\Lambda) \to \coeffs$ is an augmentation, then for any invertible elements $d_1,\dots,d_\ell \in \coeffs^\times$, the algebra map $\e': \cA(\Lambda) \to \coeffs$ defined on generators by
\[ \e'(a_i) = \frac{d_{r(a_i)}}{d_{c(a_i)}} \e(a_i) \]
is also an augmentation.
\end{lemma}

\begin{proof}
We need to check that $\e'(\partial a_i) = 0$ for each generator $a_i$ given that $\e(\partial a_i)=0$.  The differential $\partial a_i$ counts disks $u$ with positive end at $a_i$ and negative ends at Reeb chords $b_1,\dots,b_m$, and such disks contribute words $w(u)$ of the form $\pm b_1 \dots b_m$, possibly with $t_i^{\pm 1}$ generators sprinkled throughout, to $\partial a_i$.  It follows that
\[ \e'(w(u)) = \frac{d_{r(b_1)}}{d_{c(b_1)}} \ldots \frac{d_{r(b_m)}}{d_{c(b_m)}} \e(w(u)). \]
But we also know that $r(a_i) = r(b_1)$ and $c(a_i) = c(b_m)$, and that $c(b_j) = r(b_{j+1})$ for $1 \leq j \leq m-1$, so this simplifies to $\e'(w(u)) = \frac{d_{r(a_i)}}{d_{c(a_i)}} \cdot \e(w(u))$.  Summing over all $u$, we have $\e'(\partial a_i) = \frac{d_{r(a_i)}}{d_{c(a_i)}} \cdot \e(\partial a_i) = 0$.
\end{proof}

According to \cite[Proposition~3.28]{NRSSZ}, the category $\Aug(\Lambda)$ is strictly unital, with unit $e_\e = -\sum_{i=1}^\ell y_i^+$ in each $\Hom(\e,\e)$.  Given elements $\alpha \in \Hom^0(\e_1,\e_2)$ and $\beta,\gamma \in \Hom^0(\e_2,\e_1)$, we can compute the products $m_2(\beta,\alpha)$ and $m_2(\alpha,\gamma)$ from \cite[Proposition~4.14]{NRSSZ}, which determines the \dga{} of the 3-copy of $\Lambda$, and \cite[Definition~3.16]{NRSSZ}, which produces the $A_\infty$ operations from the sequence of $n$-copy \dgas{}.  In particular, if $m_2(\beta,\alpha) = e_{\e_1}$ then the coefficients $d^\alpha_i$ and $d^\beta_i$ of each $y_i^+$ in $\alpha$ and $\beta$ respectively must satisfy $d^\alpha_i \cdot d^\beta_i = 1$, and likewise if $m_2(\alpha,\gamma)=e_{\e_2}$ then $d^\alpha_i \cdot d^\gamma_i = 1$.  Moreover, given $\alpha$ with $d^\alpha_i \in \coeffs^\times$, we can repeat the argument of \cite[Proposition~5.17]{NRSSZ} to show that there are in fact some elements $\beta$ and $\gamma$ such that $m_2(\beta,\alpha) = e_{\e_1}$ and $m_2(\alpha,\gamma) = e_{\e_2}$, and that if $m_1(\alpha)=0$ then $m_1(\beta)=m_1(\gamma)=0$ and $[\beta]=[\gamma]$ in $H^0\Hom(\e_2,\e_1)$ as well.  (The construction of $\beta$ and $\gamma$ proceeds once again by induction on height once we replace the $m_2$ computations of \cite[Proposition~5.15]{NRSSZ} with the more general $m_2(y_i^+,a_j^+) = -\delta_{i,c(a_j)} a_j^+$ and $m_2(a_j^+,y_i^+) = -\delta_{i,r(a_j)}a_j^+$.)

\begin{proposition} \label{prop:unique-isomorphism}
Let $\Lambda$ be an $\ell$-component Legendrian link with one base point on each component.  Given an augmentation $\e_1: \cA(\Lambda) \to \coeffs$ and a degree-zero element of the form
\[ \alpha = d_1 y_1^+ + \dots + d_\ell y_\ell^+ + \sum_{j=1}^r k_j a_j^+, \]
where $d_i \in \coeffs^\times$ for all $i$ and $k_j \in \coeffs$ for all $j$, there is a unique augmentation $\e_2: \cA(\Lambda) \to \coeffs$ such that $\alpha$ is an isomorphism in $\Hom^0(\e_1,\e_2)$.
\end{proposition}

\begin{proof}
Given an augmentation $\e_2$, we must determine whether $m_1(\alpha)=0$ in $\Hom(\e_1,\e_2)$, and if this is the case then it will follow from the above discussion that $\alpha$ is an isomorphism.  In what follows, we will let $\e': \cA \to \coeffs$ denote the algebra homomorphism defined on generators by $\e'(a_j) = \frac{d_{r(a_j)}}{d_{c(a_j)}} \e_2(a_j)$ and $\e'(t_i) = \e_2(t_i)$; this is an augmentation by Lemma~\ref{lem:rescale-epsilon}.  We will also let $K$ denote the unique $(\e_1,\e')$-derivation defined on generators by $K(a_j) = \frac{k_j}{d_{c(a_j)}}$.

The computation of $m_1(\alpha)$ involves essentially the same computations as in \cite[Proposition~5.16]{NRSSZ}.  We compute $m_1$ according to \cite[Proposition~4.14]{NRSSZ} and \cite[Definition~3.16]{NRSSZ}:
\begin{align*}
m_1(y_i^+) &= \left(\e_1(t_i)^{-1}\e_2(t_i) - 1\right)x_i^+ + \sum_{r(a_j)=i} \e_2(a_j)\cdot a_j^+ - \sum_{c(a_j)=i} (-1)^{|a_j|}\e_1(a_j)\cdot a_j^+, \\
m_1(a_j^+) &= \sum_{\substack{a_k,b_1,\dots,b_m\\u \in \Delta(a_k;b_1,\dots,b_m)}} \sum_{\substack{1 \leq l \leq n \\ b_l=a_j}} \sigma_u \e_1(b_1\dots b_{l-1})\e_2(b_{l+1}\dots b_m) a_k^+,
\end{align*}
where $\Delta(a_j;b_1,\dots,b_m)$ denotes the moduli space of disks which are counted in $\partial a_j$ (i.e.\ rigid holomorphic disks with positive end at $a_j$ and negative ends and base points at $b_1,\dots,b_m$) and $\sigma_u$ is the sign associated to the disk $u$.  We also note that $(-1)^{|a_j|}\e_1(a_j) = \e_1(a_j)$, since $\e(a_j)=0$ unless $|a_j|=0$.  Combining these formulas, we have
\begin{align*}
m_1(\alpha) &= \sum_{i=1}^\ell (\e_1(t_i)^{-1}\e_2(t_i)-1)x_i^+ +
\sum_{j=1}^r \left( d_{r(a_j)}\e_2(a_j) - d_{c(a_j)}\e_1(a_j) \right) a_j^+ \\
&\qquad + \sum_{j=1}^r \left(\sum_{u \in \Delta(a_j;b_1,\dots,b_m)} \sum_{\substack{1\leq l \leq m\\b_l \neq t_i^{\pm 1}}} \sigma_u\e_1(b_1\dots b_{l-1})k_{b_l} \e_2(b_{l+1}\dots b_m) \right) a_j^+.
\end{align*}
In the last sum, we use $k_{b_l}$ to denote the coefficient $k_i$ such that $a_i = b_l$.

The above formula shows that we cannot have $m_1(\alpha)=0$ unless $\e_1(t_i) = \e_2(t_i)$ for all $i$.  In order to determine when the coefficient of $a_j^+$ vanishes, we examine the individual terms in the sum over disks $u$ and observe that
\begin{align*}
k_{b_l} \e_2(b_{l+1}\dots b_m) &= \frac{k_{b_l}}{d_{c(b_l)}} \cdot \left(\frac{d_{r(b_{l+1})}}{d_{c(b_{l+1})}}\e_2(b_{l+1})\right) \cdot \ldots \cdot \left(\frac{d_{r(b_m)}}{d_{c(b_m)}}\e_2(b_m)\right) \cdot d_{c(a_j)} \\
&= K(b_l)\e'(b_{l+1} \dots b_m) \cdot d_{c(a_j)},
\end{align*}
since as in the proof of Lemma~\ref{lem:rescale-epsilon} we have $c(b_i)=r(b_{i+1})$ for all $i<m$ and $c(b_m)=c(a_j)$.  Thus $m_1(\alpha)$ has $a_j^+$-coefficient equal to
\[ d_{r(a_j)}\e_2(a_j) - d_{c(a_j)}\e_1(a_j) + \sum_{u,l} \sigma(u)\e_1(b_1\dots b_{l-1}) K(b_l) \e'(b_{l+1}\dots b_m) \cdot d_{c(a_j)}, \]
which is precisely $d_{c(a_j)} \left( \e'(a_j) - \e_1(a_j) + K(\partial a_j)\right)$.  We conclude that $m_1(\alpha)=0$ in $\Hom(\e_1,\e_2)$ exactly when $\e_1 - \e' = K\circ \partial$, since this must be satisfied for all generators $a_j$ and $K$ is an $(\e_1,\e')$-derivation.  But Lemma~\ref{lem:construct-K} says that $\e_1$ and the values $\frac{k_j}{d_{c(a_j)}}$ uniquely determine $K$ and $\e'$, and in turn $\e'$ uniquely determines $\e_2$, so it follows that there is exactly one choice of $\e_2$ such that $\alpha$ is a cocycle in $\Hom(\e_1,\e_2)$.
\end{proof}

\begin{corollary}  \label{cor:count1} For any $\e \in \Aug(\Lambda; \F_q)$, the set
\[ \bigsqcup_{\e' \in \Aug(\Lambda)}  \left\{ \alpha \in \Hom^0(\e,\e') \mid m_1(\alpha) = 0, \mbox{ and } [\alpha] \mbox{ is an isomorphism in $H^0(\Hom(\e,\e'))$}\right\} \]
has cardinality $(q-1)^\ell q^{r'}$, where $\ell$ denotes the number of components of $\Lambda$ and $r'$ is the number of Reeb chords of $\Lambda$ with \dga{} degree $-1$.
\end{corollary}
\begin{proof}
Any element of $\Hom^0(\e,\e')$ which represents an invertible class in cohomology must have the form $\sum d_i y_i^+ + \beta$ with $d_1,\dots,d_\ell \in \F_q^\times$ and $\beta \in \F_q\{ a_1^+, \ldots, a_{r'}^+\}$, where the $a_i$ are the Reeb chords of degree $-1$.  Thus, for any of the $(q-1)^\ell$ choices for $d_1,\dots,d_\ell$ we have $q^{r'}$ isomorphisms from $\e$ to some $\e'$ of the form $\alpha = \sum d_i y_i^+ + \beta$, with $\e'$ uniquely determined by $\alpha$.
\end{proof}

Suppose that the Reeb chords of $\Lambda$ are sorted by height, $h(a_1) < h(a_2) < \dots < h(a_r)$, and for any $\e_1,\e_2 \in \Aug(\Lambda;\coeffs)$, consider the descending filtration of
\[ \Hom(\e_1,\e_2) = \Span_{\coeffs}\{y_1^+, \ldots, y_\ell^+, x_1^+, \ldots, x_\ell^+, a_1^+, \ldots, a_r^+ \} \]
 given by
\begin{align*}
F^{-1}\Hom(\e_1,\e_2) &= \Hom(\e_1,\e_2) \\
F^{0}\Hom(\e_1,\e_2) &= \Span_{\coeffs}\{x_1^+,\ldots, x_\ell^+, a_1^+, \ldots, a_r^+\} \\
F^{i}\Hom(\e_1,\e_2) &= \Span_{\coeffs}\{ a_i^+, a_{i+1}^+, \dots, a_r^+ \},  \quad \mbox{for $1 \leq i \leq r$}.
\end{align*}
We claim that the $A_\infty$ operations on $\Aug(\Lambda;\coeffs)$ respect this filtration in the following precise sense.
\begin{proposition}  \label{prop:mktriangular}
For any $k \geq 1$, any augmentations $\e_1, \ldots, \e_{k+1}$, and any integers $-1 \leq i_1, \ldots, i_k \leq r$, we have
\[
m_k\left( F^{i_k} \Hom(\e_k,\e_{k+1}) \otimes \cdots \otimes F^{i_1} \Hom(\e_1,\e_2)\right) \subset F^{I}\Hom(\e_1,\e_{k+1})
\]
where $I = \begin{cases} \max\{i_1, \ldots, i_k\} +1, & k \neq 2,\\ \max\{i_1, i_2\}, & k=2.\end{cases}$
\end{proposition}

\begin{proof}
We review relevant aspects of the definition of $m_k$;  see \cite[Section~3]{NRSSZ}.

The operations $m_k$ are constructed from the Chekanov--Eliashberg algebras of $\Lambda$ and of the \emph{Lagrangian projection $(k+1)$-copy} $\Lambda^{(k+1)}$ (see \cite[Section~4.2.2]{NRSSZ}) of $\Lambda$.
We consider the pure augmentation
\[ \e = (\e_1, \ldots, \e_{k+1}) : (\mathcal{A}(\Lambda^{(k+1)}), \partial) \rightarrow (\coeffs,0), \]
which produces a twisted differential $\partial^\e$ on $\mathcal{A}(\Lambda^{(k+1)})^\e$, where $\mathcal{A}(\Lambda^{(k+1)})^\e$ is now an algebra over $\coeffs$ freely generated by Reeb chords of $\Lambda^{(k+1)}$.  The Reeb chords of $\Lambda^{(k+1)}$ are enumerated as follows: each Reeb chord $a_l$ of $\Lambda$ produces Reeb chords $a^{ij}_{l}$ of $\Lambda^{(k+1)}$, where $1 \leq i,j \leq k+1$, and the remaining Reeb chords of $\Lambda^{(k+1)}$ have the form $x_l^{ij}$ and $y_l^{ij}$ for $1 \leq i < j \leq k+1$ and $1 \leq l \leq \ell$.  Moreover, we note that the construction of $\Lambda^{(k+1)}$ can be carried out so that we have height inequalities
\begin{equation}  \label{eq:k1height}
h(y_{l_1}) < h(x_{l_2}) < h(a_1) < \ldots < h(a_r)
\end{equation}
for all $1 \leq l_1,l_2 \leq \ell$ and for any choice of superscripts (omitted from notation).  There is a corresponding increasing filtration of $\alg(\Lambda^{(k+1)})^\e$,
\begin{equation}  \label{eq:incfilt}
F_{-1} \alg(\Lambda^{(k+1)})^\e \subset F_0 \alg(\Lambda^{(k+1)})^\e \subset F_1 \alg(\Lambda^{(k+1)})^\e \subset \cdots \subset F_r \alg(\Lambda^{(k+1)})^\e,
\end{equation}
where $F_{-1} \alg(\Lambda^{(k+1)})^\e$ is the sub-algebra generated by elements of the form $y^{ij}_l$; we add all generators of the form $x^{ij}_l$ to obtain $F_0\alg(\Lambda^{(k+1)})^\e$, and then to obtain $F_{p}\alg(\Lambda^{(k+1)})^\e$ from  $F_{p-1}\alg(\Lambda^{(k+1)})^\e$, for $1 \leq p \leq r$, we add all generators of the form $a_{p}^{ij}$.  As discussed above, the inequalities \eqref{eq:k1height} give $\partial^\e(F_{p}\alg(\Lambda^{(k+1)})^\e) \subset F_{p}\alg(\Lambda^{(k+1)})^\e$ for all $p\geq -1$.

Suppose we wish to determine a composition of the form $m_k(s_k^+,\dots,s_1^+)$, where each $s^+_i \in \Hom(\e_i,\e_{i+1})$ is a generator of the form $y_l^+$, $x_l^+$, or $a_l^+$ for some $l$.  There are corresponding Reeb chord generators $s^{12}_1, s^{23}_2, \ldots, s^{k,k+1}_k$ in the Chekanov--Eliashberg algebra $(\mathcal{A}^{(k+1)},\partial)$ of the $(k+1)$-copy $\Lambda^{(k+1)}$, and we have
\begin{align*}
m_k(s^+_k, \ldots, s^+_1) &= (-1)^\sigma \left( \sum_{i=1}^\ell y_i^+ \cdot \Coeff_{s^{12}_1\cdots s^{k,k+1}_k}( \partial^\e y^{1,k+1}_{i} ) +
 \sum_{i=1}^\ell x_i^+ \cdot \Coeff_{s^{12}_1\cdots s^{k,k+1}_k}( \partial^\e x^{1,k+1}_{i} ) \right.
\\
&\qquad\qquad \left. + \sum_{i=1}^r a_i^+ \cdot \Coeff_{s^{12}_1\cdots s^{k,k+1}_k}( \partial^\e a^{1,k+1}_{i} ) \right)
\end{align*}
for an appropriate sign $(-1)^\sigma$.
Showing that $\partial^\e$ preserves the filtration \eqref{eq:incfilt} translates immediately to showing that
\begin{equation} \label{eq:filtered-composition}
m_k \left( F^{i_k} \Hom(\e_k,\e_{k+1}) \otimes \cdots \otimes F^{i_1} \Hom(\e_1,\e_2)\right) \subset F^{J}\Hom(\e_1,\e_{k+1})
\end{equation}
with $J = \max\{i_1, \ldots, i_k\}$.

To improve this result when $k \neq 2$, given a generator $s^{1,k+1} \in F_{p}\alg(\Lambda^{(k+1)})^\e$, we need to examine those terms in $\partial^\e s^{1,k+1}$ which belong to $F_{p}\alg(\Lambda^{(k+1)})^\e \setminus F_{p-1}\alg(\Lambda^{(k+1)})^\e$.  An explicit computation of $(\alg(\Lambda^{(k+1)}), \partial)$ in terms of $(\alg(\Lambda),\partial)$ is given in \cite[Proposition~3.25]{NRSSZ}.  Examining it, we see that the only relevant terms are
\begin{align*}
\partial Y_l &= (Y_l)^2 \\
\partial X_l &= \Delta_l^{-1} Y_l \Delta_lX_l - X_l Y_l \\
\partial A_k &= Y_{r(k)}A_k - (-1)^{|a_k|} A_k Y_{c(k)} + \dots,
\end{align*}
where generators $y_l^{ij}$, $x_l^{ij}$, and $a_k^{ij}$ are placed into matrices $Y_l, X_l$ and $A_k$; $r(k)$ and $c(k)$ denote the components of the upper and lower endpoint of the Reeb chord $a_k$ of $\Lambda$; and $\partial$ is applied to matrices entry-by-entry.  (The $\Delta_l$ matrices are diagonal with entries corresponding to the invertible generators $t_l^{i}$, which are associated to base points.)  In passing from $(\alg(\Lambda^{(k+1)}),\partial)$ to $(\alg(\Lambda^{(k+1)})^\e,\partial^\e)$, the right hand sides of the above formulas are adjusted by replacing the $t_l^{i}$ with  $\e(t_l^i) = \epsilon_i(t_l)$ and replacing all Reeb chord generators $s$ with $s + \e(s)$.
Since $\e$ is pure, all $y^{ij}_l$ satisfy $\e(y^{ij}_l) =0$, so for a generator $s^{1,k+1} \in F_{p}\alg(\Lambda^{(k+1)})^\e \setminus F_{p-1}\alg(\Lambda^{(k+1)})^\e$, the generators in $F_{p}\alg(\Lambda^{(k+1)})^\e \setminus F_{p-1}\alg(\Lambda^{(k+1)})^\e$ which appear in $\partial^\e s^{1,k+1}$ must appear only in quadratic terms.  Thus when $k \neq 2$ we can improve \eqref{eq:filtered-composition} to obtain $J = \max\{i_1,\dots,i_k\} + 1$, as desired.
\end{proof}

Given objects $\e,\e'$ in $\Aug(\Lambda;\coeffs)$, we denote by $\Hom^0(\e,\e')^\times$ the set of all elements $\alpha \in \Hom^0(\e,\e')$ such that $m_1(\alpha)=0$ and $\alpha$ has both right and left inverses with respect to the $m_2$ operations.  As we have seen, the latter condition is equivalent to the coefficients of $y_i^+$ in $\alpha$ belonging to $\coeffs^\times$ for $1 \leq i \leq \ell$, and given $m_1(\alpha)=0$ it implies that the left and right inverses of $\alpha$ are also cocycles.  Thus $\Hom^0(\e,\e')^\times$ is nonempty if and only if $\e\cong\e'$.

The following corollary of Proposition~\ref{prop:mktriangular} would easily be shown to hold in an honest category, but it requires extra care since $\Aug(\Lambda;\coeffs)$ is an $A_\infty$ category and composition is not associative.

\begin{proposition} \label{prop:iso-bijection}
If $\e$ and $\e'$ are isomorphic objects of $\Aug(\Lambda;\coeffs)$, then there is a bijection $\Hom^0(\e,\e)^\times \cong \Hom^0(\e,\e')^\times$.
\end{proposition}

\begin{proof}
Fix $f \in \Hom^0(\e,\e')^\times$ with left inverse $g \in \Hom^0(\e',\e)^\times$.  The map $M_f: \Hom^0(\e,\e) \rightarrow \Hom^0(\e,\e')$, defined by $M_f(\beta)= m_2(f,\beta)$, and the similarly defined $M_g: \Hom^0(\e,\e') \rightarrow \Hom^0(\e,\e)$ are vector space isomorphisms.  To see this, we compute using the $A_\infty$ relations that
\begin{align*} M_g \circ M_f(\beta) &= m_2(g,m_2(f,\beta)) \\
&= m_2(m_2(g,f),\beta) + m_1(m_3(g,f,\beta)) \\
& \qquad + m_3(m_1(g),f,\beta) + (-1)^{|g|}m_3(g,m_1(f),\beta) + (-1)^{|g|+|f|}m_3(g,f,m_1(\beta)) \\
&= \beta + m_1(m_3(g,f,\beta)) + m_3(g,f,m_1(\beta)).
\end{align*}
Applying Proposition \ref{prop:mktriangular}, we see that the matrix of $M_g \circ M_f$ with respect to the basis
\[ \{y_1, \ldots, y_\ell, a^+_1, \ldots, a^+_r\}\setminus \{a^+_i \mid |a^+_i| \neq 0\} \]
 has the form $I + N$ where $N$ is a strictly lower triangular matrix.  Thus $M_g \circ M_f$ is injective, and in particular the map $M_f: \Hom^0(\e,\e) \to \Hom^0(\e,\e')$ is injective.  By the same argument, we also see that $M_g: \Hom^0(\e,\e') \to \Hom^0(\e,\e)$ is injective.

In addition, we claim that $M_f\left( \Hom^0(\e,\e)^\times \right) \subset \Hom^0(\e,\e')^\times$.  Indeed, given $\alpha \in \Hom^0(\e,\e)^\times$, we have $m_1(\alpha)=0$ and $m_1(f)=0$ by definition, so
\[
m_1(M_f(\alpha)) = m_2(m_1(f),\alpha) + (-1)^{|f|}m_2(f, m_1(\alpha)) = 0.
\]
Also, the coefficients of the $y^+_l$ in $m_2(f,\alpha)$ are products of the corresponding coefficients of $f$ and $\alpha$, and in particular are all units, so that $m_2(f,\alpha)$ has left and right inverses.  This establishes that $M_f(\alpha) \in \Hom^0(\e,\e')^\times$.  Similarly, $M_g$ sends $\Hom^0(\e,\e')^\times$ to $\Hom^0(\e,\e)^\times$.

Since both $M_f$ and $M_g$ are injective as maps between $\Hom^0(\e,\e)^\times$ and $\Hom^0(\e,\e')^\times$, it follows that $\Hom^0(\e,\e)^\times$ is in bijection with $\Hom^0(\e,\e')^\times$.
\end{proof}

\begin{remark}
The same argument works to show that $\Hom^0(\e,\e)^\times \cong \Hom^0(\e',\e)^\times$, so repeated application of Proposition~\ref{prop:iso-bijection} shows that if $\e_0,\e_1,\e'_0,\e'_1: \cA(\Lambda)\to\coeffs$ are all isomorphic then there is a bijection $\Hom^0(\e_0,\e_1)^\times \cong \Hom^0(\e'_0,\e'_1)^\times$.
\end{remark}

\begin{corollary} \label{cor:count2}
Fix $\epsilon: \cA(\Lambda) \to \F_q$.  Write $B^0(\epsilon, \epsilon)$ for the elements
of $\Hom^{0}(\epsilon, \epsilon)$ which are coboundaries.  Then
\[ \# \{\epsilon' : \cA(\Lambda) \to \F_q \mid \epsilon' \cong \epsilon\} = q^{\dim \Hom^{0}(\epsilon, \epsilon) - \dim B^0(\epsilon, \epsilon) - \ell} \cdot
\frac{(q-1)^\ell}{\lvert\Aut(\epsilon)\rvert}. \]
\end{corollary}

\begin{proof}
Combining Corollary~\ref{cor:count1} and Proposition~\ref{prop:iso-bijection}, we have shown that
\[ (q-1)^\ell q^{r'} = \left\lvert\bigsqcup_{\e' \cong \e} \Hom^0(\e,\e')^\times\right\rvert
= \#\{\e' \mid \e'\cong\e\} \cdot \left\lvert\Hom^0(\e,\e)^\times\right\rvert. \]
But $\lvert\Hom^0(\e,\e)^\times\rvert$ is $\lvert\Aut(\e)\rvert$ times the number of coboundaries in $\Hom^0(\e,\e)$, which is $q^{\dim B^0(\e,\e)}$.  The desired formula follows once we note in addition that $r' = \dim \Hom^0(\e,\e)-\ell$.
\end{proof}

We are now in a position to prove Theorem~\ref{thm:card}.

\begin{proof}[Proof of Theorem \ref{thm:card}]
Fix $\epsilon: \cA(\Lambda) \to \F_q$.
We expand the quantity $\dim \Hom^{0}(\epsilon, \epsilon) - \dim B^0(\epsilon, \epsilon)$.
By definition, we have:
$$\dim \Hom^{i}(\epsilon, \epsilon) = \dim B^{i}(\epsilon, \epsilon) + \dim H^i \Hom(\epsilon, \epsilon) + \dim B^{i+1}(\epsilon, \epsilon).$$
Summing over $i<0$,
$$\sum_{i < 0} (-1)^i \dim \Hom^{i}(\epsilon, \epsilon) = - \dim B^0(\epsilon, \epsilon) + \sum_{i < 0} (-1)^i \dim H^i \Hom(\epsilon, \epsilon).$$
Note that we also have
$$\chi_* =  \sum_{i < 0} (-1)^{i+1} r_i  + \sum_{i \ge 0} (-1)^i r_i =
-2\ell + \sum_{i \le 0} (-1)^i \dim \Hom^{i}(\epsilon, \epsilon) - \sum_{i > 0} (-1)^{i} \dim \Hom^{i}(\epsilon, \epsilon)$$
and
$$-tb = \sum_{i} (-1)^{i} \dim \Hom^{i}(\epsilon, \epsilon),$$
and hence
\begin{align*}
\frac{\chi_* - tb}{2} & =  -\ell + \dim \Hom^0(\epsilon, \epsilon) + \sum_{i < 0} (-1)^{i} \dim \Hom^{i}(\epsilon, \epsilon) \\
& =
\dim \Hom^0(\epsilon, \epsilon) - \dim B^0(\epsilon, \epsilon)
- \ell + \sum_{i < 0} (-1)^i \dim H^i \Hom(\epsilon, \epsilon).
\end{align*}
Thus we can rewrite the result of Corollary \ref{cor:count2} as:
 $$\# \{\epsilon' : \cA \to \F_q \mid \epsilon' \cong \epsilon\} =
q^{\frac{\chi_* - tb}{2}}  q^{-\sum_{i < 0} (-1)^i \dim H^i \Hom(\epsilon, \epsilon)} \cdot
\frac{(q-1)^\ell}{\lvert\Aut(\epsilon)\rvert}.$$
Rearranging, we have
$$
\frac{| H^{-1} \Hom(\epsilon, \epsilon) | \cdot | H^{-3} \Hom(\epsilon, \epsilon) |
\cdots}{\lvert\Aut(\epsilon)\rvert \cdot  |H^{-2} \Hom(\epsilon, \epsilon)| \cdots} =  q^{\frac{tb - \chi_*}{2}} \cdot (q-1)^{-\ell}
\cdot
\# \{\epsilon' : \cA \to \F_q \, | \, \epsilon' \cong \epsilon\},$$
and summing over the isomorphism classes on both sides gives the assertion of Theorem \ref{thm:card}.
\end{proof}

\section{Doubly Lagrangian slice knots}
\label{sec:doubly-slice}

Given Legendrian knots $\Lambda_-, \Lambda_+ \subset \R^3$, we say that $\Lambda_-$ is \emph{Lagrangian concordant} to $\Lambda_+$, denoted $\Lambda_- \prec \Lambda_+$, if there is a Lagrangian cylinder $L$ in the symplectization $(\R \times \R^3, d(e^t(dz-y\,dx)))$ and some $T \gg 0$ such that
\[ L \cap (\{t\} \times \R^3) = \begin{cases} \{t\} \times \Lambda_-, & t \leq -T \\ \{t\} \times \Lambda_+, & t \geq T. \end{cases} \]
The Lagrangian $L$ in this case is called a \textit{Lagrangian
  concordance} from $\Lambda_-$ to $\Lambda_+$.
If $U$ is the standard Legendrian unknot with $tb(U)=-1$, and we have both $U \prec \Lambda$ and $\Lambda \prec U$, we will say that $\Lambda$ is \emph{doubly Lagrangian slice}.  In this section we will use Corollary~\ref{cor:ruling-polynomial} to understand $\Aug(\Lambda)$ for such $\Lambda$.

\begin{theorem}[\cite{CNS}]
If $\Lambda$ is doubly Lagrangian slice, then its ruling polynomial\footnote{In \cite{CNS} the stated theorem is $R_\Lambda(z)=1$, but we use the normalization $R_\Lambda(z) = \sum_{R\in\cR} z^{-\chi(R)}$ rather than the sum $\sum z^{1-\chi(R)}$ of \cite{CNS}.} satisfies $R_\Lambda(z) = z^{-1}$.
\end{theorem}

\noindent
A doubly Lagrangian slice knot must have $tb(\Lambda)=-1$, since the Thurston--Bennequin number is invariant under Lagrangian concordance \cite{Cha}, so according to Corollary~\ref{cor:ruling-polynomial} such knots satisfy
\[ \#\pi_{\geq 0}\Aug(\Lambda;\F_q)^* = q^{-1/2}(q^{1/2} - q^{-1/2})^{-1} = \frac{1}{q-1}. \]

Now given a Lagrangian concordance $L$ from $\Lambda_-$ to
$\Lambda_+$, the work of \cite{EHK} produces a \dga{} morphism
\[
\Phi_L :\thinspace \cA(\Lambda_+) \to \cA(\Lambda_-)
\]
between the \dgas{} of $\Lambda_\pm$. In \cite{EHK}, these \dgas{} are
considered over $\Z/2$, but in fact it is known that the result still holds for the
\dgas{} over $\Z[t,t^{-1}]$ as in \cite{ENS}. More precisely, the
cobordism map $\Phi_L$ as presented in \cite{EHK} counts gradient flow
trees for a Morse cobordism, and these trees in turn correspond to
holomorphic disks for Legendrian contact homology in
$J^1([0,1]\times\R)$, the $1$-jet space of a strip, by
\cite{Ekholmtrees}. Orientations for the moduli spaces of these disks
have been worked out in \cite{EESorient}, and it follows that we can
lift $\Phi_L$ from a map over $\Z/2$ to a map over $\Z$. To lift this
further to $\Z[t,t^{-1}]$, choose base points on $\Lambda_\pm$ and
connect them by a path $\gamma$ in $L$; then the powers of $t$ in the
lifted map $\Phi_L$ count signed intersections
of the flow trees with $\gamma$.

By \cite[Proposition~3.29]{NRSSZ}, the map $\Phi_L$ of \dgas{} over
$\Z[t,t^{-1}]$ induces a functor
\[ F_L: \Aug(U;\F_q) \to \Aug(\Lambda;\F_q) \]
of unital $A_\infty$ categories. Let $\epsilon_U$ be the unique augmentation in $\Aug(U;\F_q)$, and write $\epsilon_L$ for the augmentation of $\cA(\Lambda)$ that is the image of $\epsilon_U$ under $F_L$; note that $\e_L = \e_U \circ \Phi_L$ and $\e_L$ is the augmentation induced by capping off the concave end of $L$ to get a Lagrangian disk filling of $\Lambda$.

\begin{proposition} \label{prop:slice-homotopic}
If $\Lambda$ is doubly Lagrangian slice, then all augmentations $\e: \cA(\Lambda) \to \F_q$ are homotopic to $\e_L$.
\end{proposition}

\begin{proof}

Since $\e_L$ is induced by a disk filling of $\Lambda$, \cite[Proposition~5.7]{NRSSZ} says that $H^*\Hom(\e_L,\e_L) \cong H^*(D^2; \F_q)$.  In particular, the cohomology of $\Hom(\e_L,\e_L)$ is supported in degree zero and has rank 1, so the only possible automorphisms of $\e_L$ are nonzero scalar multiples of the unit $e_{\e_L}$ and there are no higher homotopies.  The isomorphism class of $\e_L$ therefore contributes $\frac{1}{|\F_q^\times|} = \frac{1}{q-1}$ to the homotopy cardinality of $\Aug(\Lambda;\F_q)^*$.  But we have seen above that this is the entire homotopy cardinality, so all augmentations of $\Lambda$ are isomorphic to $\e_L$.  It follows by \cite[Proposition~5.17]{NRSSZ} that all augmentations of $\Lambda$ are homotopic to $\e_L$.
\end{proof}

\begin{remark}
In fact, this can be strengthened somewhat, using additional results
from \cite{CNS}: if $U\prec \Lambda$ and $\Lambda \prec U$ then the
Legendrian $n$-stranded cable $S(\Lambda,tw_n)$ (see
\cite[Section~2.4]{CNS}), which is topologically the Seifert-framed
$n$-cable of $\Lambda$, has ruling polynomial $z^{-n}$ and $tb=-n$,
hence $\#\pi_{\ge 0}\Aug(S(\Lambda,tw_n);\F_q)^\ast = \frac{1}{(q-1)^n}$.
Given a Lagrangian concordance $L$ from $U$ to $\Lambda$, we produce a
Lagrangian concordance $L^n$ (i.e.\ a disjoint union of $n$ Lagrangian
cylinders) from $S(U,tw_n)$ to $S(\Lambda,tw_n)$, and since
$S(U,tw_n)$ is Legendrian isotopic to $n$ unlinked copies of $U$, it admits
a filling by $n$ Lagrangian disks $\bigsqcup_{i=1}^n D^2$.  Let $\e_U$ be the induced augmentation;
then $\e_{L^n} = F_{L^n}(\e_U)$ corresponds to the filling
$\left(\bigsqcup D^2\right) \cup_{S(\Lambda,tw_n)} L^n$ of $S(\Lambda,tw_n)$.  It follows that
\[ H^*\Hom(\e_U,\e_U) \cong H^*\Hom(\e_{L^n},\e_{L^n}) = H^*(D^2 \sqcup \dots \sqcup D^2; \F_q) \cong \F_q^{\oplus n} \]
as abelian groups, supported in degree 0.  If we can show that
$H^*\Hom(\e_{L^n},\e_{L^n}) \cong \F_q^{\oplus n}$ as a ring, then it
will follow that $\lvert\Aut(\e_{L^n})\rvert = (q-1)^n$, and hence
that all augmentations of $S(\Lambda,tw_n)$ are homotopic to
$\e_{L^n}$ just as in the proof of
Proposition~\ref{prop:slice-homotopic}.
This is a special case of the statement that the isomorphism $H^*\Hom(\e_L,\e_L) \cong H^*(L)$ is a ring isomorphism for any exact Lagrangian filling $L$, which should be provable by considering $Y$-shaped gradient flow trees on $L$; here we prove just the special case that we need.

If $L'$ is a Lagrangian concordance from $\Lambda$ to $U$ and $(L')^n$ is the corresponding cable of $L'$, then $L \cup_\Lambda L'$ is Hamiltonian isotopic to a product concordance \cite{EP,Cha-symmetric}, hence so is $L^n \cup_{S(\Lambda,tw_n)} (L')^n$.  The ring homomorphism
\[ F_{(L')^n} \circ F_{L^n} = F_{L^n \cup_{S(\Lambda,tw_n)} (L')^n}: H^*\Hom(\e_U,\e_U)\to H^*\Hom(\e_U,\e_U) \]
is therefore an isomorphism, and so
$F_{L^n}: H^*\Hom(\e_U,\e_U) \to H^*\Hom(\e_{L^n},\e_{L^n})$
is injective.  Both of these rings have cardinality $q^n$, so the map is bijective and hence an isomorphism of rings; and a direct computation shows that $H^*\Hom(\e_U,\e_U) = \F_q^{\oplus n}$ as rings, so the same is true of $H^*\Hom(\e_{L^n},\e_{L^n})$ as claimed.
\end{remark}

By Proposition~\ref{prop:slice-homotopic}, if $\Lambda$ is doubly
Lagrangian slice, then $F_L: \Aug(U;\F_q) \to \Aug(\Lambda;\F_q)$ maps
between categories that each have one object up to equivalence. In
fact, we can say more. Say that an $A_\infty$ functor between
unital $A_\infty$ categories is an \textit{$A_\infty$ quasi-equivalence} if the induced
functor on cohomology categories is fully faithful and essentially surjective.

\begin{theorem} \label{thm:slice-equivalence}
If $\Lambda$ is doubly Lagrangian slice and $L$ is a Lagrangian
concordance from $U$ to $\Lambda$, then the induced functor $F_L:
\Aug(U;\F_q) \to \Aug(\Lambda; \F_q)$ is an $A_\infty$ quasi-equivalence.
\end{theorem}

\begin{proof}
We revisit the proof of Proposition~\ref{prop:slice-homotopic}.  Since
all objects of $\Aug(\Lambda)$ are isomorphic to $\e_L$, the functor
$F_L$ is essentially surjective.  Moreover, the induced map
$H^*\Hom(\e_U,\e_U) \to H^*\Hom(\e_L,\e_L)$ is an isomorphism, since
$F_L(e_{\e_U}) = e_{\e_L}$ and the cohomology rings are spanned by the
classes $[e_{\e_U}]$ and $[e_{\e_L}]$ respectively, so $F_L$ is also
fully faithful on cohomology.
\end{proof}

In light of Theorem~\ref{thm:slice-equivalence}, a natural question to
ask is whether the underlying \dga{} map $\Phi_L$ is a quasi-isomorphism
(or, better yet, a stable tame isomorphism, which would automatically
imply Theorem~\ref{thm:slice-equivalence}). This is currently an open
question, but Theorem~\ref{thm:slice-equivalence} does imply that the
map $\Phi_L$ is a quasi-isomorphism on a certain completion of the
\dgas{}.

Given an augmentation $\epsilon :\thinspace \cA(\Lambda) \to \F_q$,
we follow the notation of \cite{NRSSZ}: let
$\cA^\epsilon = (\cA\otimes\F_q)/(t=\epsilon(t))$
(this is the unital tensor algebra over $\F_q$ generated by Reeb chords of $\Lambda$, and the differential $\partial$ descends to $\cA^\epsilon$),
and let $\cA^\epsilon_+$ denote the subspace of $\cA^\epsilon$ generated by all nonempty words of Reeb chords.
Define $\phi_\e :\thinspace \cA^\epsilon \to \cA^\epsilon$ to be the algebra automorphism determined by $\phi_\e(a) = a+\e(a)$ for all Reeb chords $a$. Then the differential
$\partial_\e := \phi_\e \circ \partial \circ \phi_\e^{-1}$ on $\cA^\e$ is filtered with respect to the filtration
\[
\cA^\e \supset \cA^\e_+ \supset (\cA^\e_+)^2 \supset \cdots.
\]
Let $\widehat{\cA}^\epsilon(\Lambda)$ denote the completion of
$\cA^\epsilon$ with respect to this filtration, and note that $\partial_\e$ extends to a differential on $\widehat{\cA}^\epsilon(\Lambda)$.

\begin{proposition}
If $\Lambda$ is doubly Lagrangian slice and $L$ is a Lagrangian
concordance from $U$ to $\Lambda$,
\label{prop:completion}
then the morphism
\[
\Phi_L :\thinspace \widehat{\cA}^{\epsilon_L}(\Lambda) \to \widehat{\cA}^{\epsilon_U}(U)
\]
is a quasi-isomorphism.
\end{proposition}

\begin{proof}
When $(\cA,\partial)$ is the \dga{} for a Legendrian knot and $\e:\thinspace \cA\to\F_q$ is an augmentation, the twisted \dga{} $(\cA^\e,\partial_\e)$ dualizes to the $A_\infty$ algebra $\Hom_-(\e,\e)$, where $\Hom_-$ denotes morphism in the Bourgeois--Chantraine augmentation category $\AugBC(\Lambda;\F_q)$ \cite{BC,NRSSZ}.
Note that $\Hom_-(\e,\e)$ is a subcomplex of $\Hom_+(\e,\e)$, and that in our setting the $A_\infty$ functor $F_L :\thinspace \Aug(U;\F_q) \to \Aug(\Lambda;\F_q)$ in particular gives chain maps $F_L \thinspace \Hom_\pm(\e_U,\e_U) \to \Hom_\pm(\e_L,\e_L)$. We then have the following commutative diagram:
\[
\xymatrix{
0 \ar[r] & \Hom_-(\e_U,\e_U) \ar[r] \ar^{F_L}[d] & \Hom_+(\e_U,\e_U) \ar[r] \ar_{F_L}[d] & C^*(U) \ar[r] \ar[d] & 0 \\
0 \ar[r] & \Hom_-(\e_L,\e_L) \ar[r] & \Hom_+(\e_L,\e_L) \ar[r] & C^*(\Lambda) \ar[r] & 0,
}
\]
where the exact rows are as in \cite[Proposition~5.2]{NRSSZ}.

By Theorem~\ref{thm:slice-equivalence}, the middle column induces an isomorphism on cohomology, and so the left column does as well. That is, $F_L :\thinspace \AugBC(U;\F_q) \to \AugBC(\Lambda;\F_q)$ is an $A_\infty$ quasi-equivalence.
The result now follows by noting that the duals of the $A_\infty$ algebras $\AugBC(U;\F_q)$, $\AugBC(\Lambda;\F_q)$ are the completed \dgas{}
$(\widehat{\cA}^{\e_U}(U),\partial_{\e_U})$ and $(\widehat{\cA}^{\e_L}(\Lambda),\partial_{\e_L})$, and that the dual map
to $F_L :\thinspace \AugBC(U;\F_q) \to \AugBC(\Lambda;\F_q)$ is $\Phi_L$.
\end{proof}

\begin{remark}
Although we would like to replace the completions $\widehat{\cA}$ by the original algebras $\cA$ in the statement of Proposition~\ref{prop:completion}, our technique does not allow us to do this directly. Indeed, it is possible for two (augmented) \dgas{} $(\cA,\partial)$ and $(\cA',\partial')$ to have dual $A_\infty$ algebras that are $A_\infty$ quasi-isomorphic without the original \dgas{} being quasi-isomorphic. Let $\coeffs$ be a field, let $\cA = \coeffs$ with $\partial = 0$, and let $\cA'$ be the tensor algebra over $\coeffs$ generated by $a,b$ with $\partial'(a) = b-b^2$, $\partial'(b) = 0$, $|a|=1$, $|b|=0$. The $A_\infty$ algebra dual to $(\cA',\partial')$ has generators $a^*,b^*$ with $m_1(b^*) = a^*$, $m_1(a^*)=0$, and thus the $m_1$ cohomologies of the $A_\infty$ algebras dual to both $(\cA,\partial)$ and $(\cA',\partial')$ are $0$. The inclusion map $i :\thinspace \cA \to \cA'$ thus induces an isomorphism on cohomology on the dual $A_\infty$ algebras. However, $H_0(\cA,\partial) = \coeffs$, while
$H_0(\cA',\partial') \cong \coeffs[b]/(b-b^2)$, and so $i_* :\thinspace H(\cA,\partial) \to H(\cA',\partial')$ is not an isomorphism in this case.
\end{remark}

\section{Counting $2m$-graded augmentations}
\label{sec:2m-graded}

So far, we have restricted ourselves to the setting where the \dga{} has a $\Z$ grading and augmentations and rulings respect this grading.
In this section, we sketch how to extend our results on counting
augmentations to the more general setting of a $\Z/(2m)$ grading for
some integer $m$. By a result of Sabloff \cite[Theorem 1.3]{Sab},
$\Lambda$ can have $2m$-graded augmentations and rulings only if the
rotation number $r(\Lambda)$ is $0$. In this section, we further
assume that \textit{each component} of $\Lambda$ has rotation number
$0$; this is no additional condition if $\Lambda$ is connected, but is
necessary to define the shifted Euler characteristic $\chi_*$
below. In addition, if $\Lambda$ has multiple components, we equip
each component of $\Lambda$ with a Maslov potential, in order to
define the degrees of Reeb chords between different components of
$\Lambda$ and thus the grading on the augmentation category; in fact,
this is implicitly done in the $\Z$-graded setting as well.

Let $m\geq 0$,
and fix a finite field $\F_q$ as before. As described in
\cite{NRSSZ}, we can construct the $\Z/(2m)$-graded augmentation
category of $\Lambda$, which we again write as $\Aug(\Lambda;\F_q)$,
whose objects are $2m$-graded augmentations (i.e., $\epsilon
:\thinspace \cA(\Lambda) \to \F_q$ with $\epsilon(a) = 0$ unless $|a|
\equiv 0 \pmod{2m}$), and whose morphisms are $\Z/(2m)$-graded vector
spaces over $\F_q$. We write $\cR$ for the set of $2m$-graded normal
rulings of $\Lambda$, and define the $2m$-graded ruling polynomial $R^{2m}_\Lambda(z) =
\sum_{R\in\cR} z^{-\chi(R)}$. When $m=0$, this recovers the
definitions from before. A case of particular interest is
$m=1$, where $zR^2_\Lambda(z)$ is the coefficient of $a^{-tb(\Lambda)-1}$ in the HOMFLY-PT polynomial $P_\Lambda(a,z)$, as shown in \cite{Rut}.

Since the naive number of augmentations changes under stabilization of the \dga{}, we need a normalized augmentation count that can then be related to the ruling polynomial as in Equation~\eqref{eq:hr} from the Introduction. For $m=0$, the normalization is given by $q^{-\chi_*(\Lambda) / 2} \# \{ \cA(\Lambda) \to \F_q \}$ where $\chi_*$ is the shifted Euler characteristic as defined in the Introduction; a similar expression is given in \cite{NS} in the {\em odd}-graded case. The key in the $\Z$-graded case ($m=0$) is that if we stabilize $\cA$ by adding two generators of degrees $i$ and $i-1$, then $\# \{ \cA\to\F_q\}$ is multiplied by $q$ when $i=0$ and is unchanged otherwise, while $\chi_*(\Lambda)$ increases by $2$ when $i=0$ and is unchanged otherwise.

Now let $m>0$; to obtain a similar result in the $\Z/(2m)$-graded
case, we modify the definition of $\chi_*$ as follows.
Since each component of $\Lambda$ has rotation number $0$ by assumption,
$\cA(\Lambda)$ is $\Z$-graded, and we can set $r_i$ to be the number of Reeb chords of degree $i$ as before.
Now let
\[
s_k = r_{2mk}-r_{2mk+1}+\cdots+r_{2mk+2m-2}-r_{2mk+2m-1},
\]
and define
\[
\chi_*(\Lambda) = \sum_k (2k+1)s_k.
\]
As in the $m=0$ case, this has the property that stabilizing $\cA$ by adding generators of degree $i$ and $i-1$ increases $\chi_*$ by $2$ if $i\equiv 0\pmod{2m}$ and leaves $\chi_*$ unchanged otherwise. Since $\#\{\cA(\Lambda)\to\F_q\}$, the number of $2m$-graded augmentations $\cA\to\F_q$, is multiplied by $q$ in the first case and is unchanged in the second, it follows that
$q^{-\chi_*(\Lambda)/2}\#\{\cA(\Lambda)\to\F_q\}$ is invariant under stable tame isomorphism.

A result from
 \cite{HR} together with a nontrivial extension of Lemma 5 from \cite{NS} (the extension is stated here as Lemma~\ref{lem:returnlemma}, and we limit ourselves to a sketch of the proof) can be used to show the following.

\begin{proposition}
For any $m\geq 0$, we have \label{prop:2mrulings}
\begin{equation} \label{eq:2mrulings}
q^{-\chi_*(\Lambda) / 2} (q-1)^{-\ell}  \cdot  \# \{ \cA(\Lambda) \to \F_q \} =R^{2m}_\Lambda(q^{1/2}-q^{-1/2}).
\end{equation}
\end{proposition}

\noindent
Note that for $m=0$, this is Equation~\eqref{eq:hr} from the Introduction.

\begin{proof}
Given a $2m$-graded normal ruling $R$ of $\Lambda$, those crossings with degree $0 \pmod{2m}$ are divided into three classes:  \emph{departures}, \emph{returns}, and \emph{switches}.  At a departure (resp.\ return) $c$,  the two disks whose boundaries meet at $c$ satisfy the normality condition, i.e., they are disjoint or nested, at values of  $x$  immediately preceding (resp.\ following) $c$.  Now \cite[Theorem 3.4]{HR} shows that (under the assumption that $\Lambda$ is the resolution of a plat position front diagram, which can always be made after a Legendrian isotopy) the number of $2m$-graded augmentations is
\[
\# \{ \cA(\Lambda) \to \F_q \} = \sum_{R \in \mathcal{R}}(q-1)^{-\chi(R) + \ell} \cdot q^{\# \mathrm{returns}}.
\]
The equality \eqref{eq:2mrulings} then follows in a straightforward manner from the following generalization of \cite[Lemma 5]{NS}.
\end{proof}

\begin{lemma}
Suppose $\Lambda$ is the resolution of a
front diagram in plat position.
\label{lem:returnlemma}
For any $2m$-graded ruling $R$ of $\Lambda$, we have
\begin{equation} \label{eq:returnlemma}
\# \mathrm{returns} = \frac{1}{2}\left( \chi_*(\Lambda) + \chi(R)\right).
\end{equation}
\end{lemma}
\begin{proof}[Sketch of proof.]
For $k \in \Z$, let $C_k$ denote the number of crossings on the front projection of $\Lambda$ with degree $k$, and when $k \equiv 0 \pmod 2m$, write
\[
C_k = D_k + R_k + S_k
\]
where $D_k$, $R_k$, and $S_k$ respectively denote the number of departures, returns, and switches of degree $k$.  Since $\Lambda$ is the resolution of a front diagram, the Reeb chords of $\Lambda$ are in one-to-one correspondence with crossings and right cusps of the front of $\Lambda$, and so
\[
r_k = \begin{cases} C_k & k\neq 1, \\ C_k + \# \mathrm{right\ cusps} & k =1; \end{cases}
\]
along with the fact that $\chi(R) =   \#\mathrm{right\ cusps}- \#\mathrm{switches}$,  this implies that
\eqref{eq:returnlemma} is equivalent to:
\begin{equation}
0 = \sum_{k \equiv 0\!\!\!\!\!\pmod{2m}} \left(  \left(\frac{k}{m}  + 1 \right) D_k + \left(\frac{k}{m}  - 1 \right) R_k + \left(\frac{k}{m}  \right)S_k \right) +
\sum_{k \not\equiv 0\!\!\!\!\!\pmod{2m}} (-1)^k\left(  2\left\lfloor \frac{k}{2m} \right\rfloor + 1 \right) C_k.
\label{eq:returnequiv}
\end{equation}
To establish \eqref{eq:returnequiv}, we associate an integer $I(x)$ to values of $x\in\mathbb{R}$ that are not the $x$ coordinates of crossings or cusps of the front projection of $\Lambda$.  The function $I$ will have the properties:
\begin{itemize}
\item[(i)] As $x$ moves from left to right in $\mathbb{R}$, $I(x)$ is constant except when $x$ passes a crossing of $\Lambda$.
\item[(ii)] $I(x)=0$ for $x\ll 0$ and $x\gg 0$.
\item[(iii)]  Suppose that a single crossing $c$ occurs in the interval $(x_0,x_1)$. Then:
\[
I(x_1) -I(x_0) = \begin{cases}  \frac{|c|}{m}+1 & \mbox{if $c$ is a departure,} \\[0.5ex]
\frac{|c|}{m}-1 & \mbox{if $c$ is a return,} \\[0.5ex]
\frac{|c|}{m} & \mbox{if $c$ is a switch,} \\[0.5ex]
(-1)^{|c|}\left( 2 \left\lfloor \frac{|c|}{2m}\right\rfloor+1\right) & \mbox{if $|c| \not\equiv 0\pmod{2m}$.} \end{cases}
\]
\end{itemize}
The equation \eqref{eq:returnequiv} then follows from considering the total change in $I$
from $x\ll 0$ to $x\gg 0$.

To define $I$ we introduce some notation.  Let $D_1, \ldots, D_n$ denote the disks that constitute the ruling $R$.  Each $D_i$ has an upper strand and a lower strand;  we denote the value of the ($\Z$-valued) Maslov potential on the upper (resp. lower) strand of $D_i$ by $a_i$ (resp. $b_i$).  Note that these values
may change along the upper and lower strands of $D_i$ but only by multiples of $2m$.

We now define
\[
I(x) = I_1(x) + I_2(x) +I_3(x),
\]
where $I_1,I_2,I_3$ are given as follows.
The first term is defined by
\[
I_1(x) = \sum_{\labellist
\small
\pinlabel $a_i$ [r] at -2 96
\pinlabel $b_i$ [r] at -2 32
\pinlabel $a_j$ [l] at 128 64
\pinlabel $b_j$ [l] at 128 0
\endlabellist
\includegraphics[scale=.2]{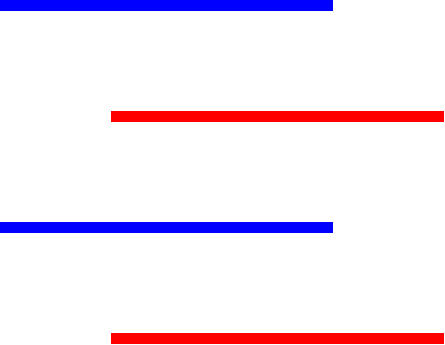} } \quad (-1)^{b_i-a_j} \left( 2 \left\lfloor \frac{b_i-a_j}{2m} \right\rfloor + 1 \right)
\]
where the sum is over all pairs of interlaced disks at $x$, $D_i$ and $D_j$, with $D_i$ the upper of the two disks.  Next,
\[
I_2(x) = \sum_{\labellist
\small
\pinlabel $a_j$ [r] at -2 96
\pinlabel $a_i$ [l] at 128 64
\pinlabel $b_i$ [l] at 128 32
\pinlabel $b_j$ [r] at -2 0
\endlabellist
\includegraphics[scale=.2]{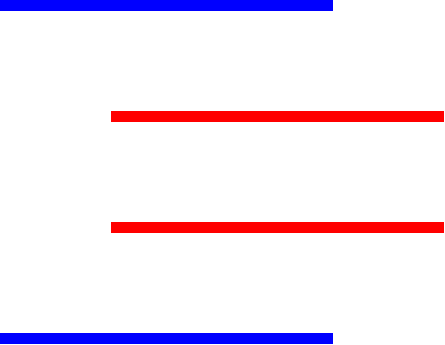} } \quad 2 (-1)^{a_i-a_j}  \left\lfloor \frac{a_i-b_i}{2m} \right\rfloor
\]
where the sum is over all pairs of nested disks, $E_i$ and $E_j$, with $E_i$ inside $E_j$ at $x$.  Finally,
\[
I_3(x) =  \sum_{\labellist
\small
\pinlabel $a_i$ [r] at -2 32
\pinlabel $b_i$ [r] at -2 0
\endlabellist
\includegraphics[scale=.2]{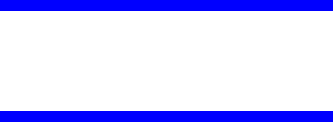} } \quad  \left\lfloor \frac{a_i-b_i}{2m} \right\rfloor
\]
where the sum is over all disks at $x$.

That the properties (i) and (ii) hold is easily verified.  That $I(x)$ satisfies (iii) can be verified by a lengthy but fairly straightforward check based on considering the combinatorics of the ruling disks that meet at a crossing $c$ in a case-by-case manner.  When $c$ is a switch, one also needs to consider the possible ways in which a third disk $D_k$ may intersect $D_i$ and $D_j$.
\end{proof}

Proposition~\ref{prop:2mrulings} relates the naive count of augmentations $\#\{\cA(\Lambda)\to\F_q\}$ to the $2m$-graded ruling polynomial.
On the other hand, the same argument as in Section~\ref{sec:counting}, but now in the $2m$-graded setting, yields the following:
\[
\#\{\cA(\Lambda)\to\F_q\} = \sum_{[\epsilon]} \frac{(q-1)^\ell}{\lvert\Aut(\epsilon)\rvert} q^{\dim\Hom^{0}(\epsilon,\epsilon)-\dim B^{0}(\epsilon,\epsilon)-\ell}.
\]
Here we note that all gradings are taken in $\Z/(2m)$, and that the sum is over isomorphism classes of augmentations in the $2m$-graded category.
Note in particular that the exponent
$\dim\Hom^{0}(\epsilon,\epsilon)-\dim B^{0}(\epsilon,\epsilon)-\ell$
in the summand is independent of the choice of representative $\epsilon$ of the isomorphism class $[\epsilon]$, cf.\ Corollary~\ref{cor:count2}.

Combined with Proposition~\ref{prop:2mrulings}, this yields the following.

\begin{proposition}
For any $m\geq 0$, we have \label{prop:2mcard}
\[
\sum_{[\epsilon]} \frac{1}{\lvert\Aut(\epsilon)\rvert} q^{\dim\Hom^{0}(\epsilon,\epsilon)-\dim B^{0}(\epsilon,\epsilon)-\ell
-(\chi_*(\Lambda)-tb(\Lambda))/2} = q^{tb(\Lambda)/2} R_\Lambda(q^{1/2}-q^{-1/2}).
\]
\end{proposition}

\noindent
By comparison with Corollary~\ref{cor:ruling-polynomial}, it then
makes sense to regard the left hand side of the equation in
Proposition~\ref{prop:2mcard} as being the homotopy cardinality of the
$2m$-graded augmentation category for $m\geq 0$. Indeed, for $m=0$,
the proof of Theorem~\ref{thm:card} shows that
\[
q^{\dim\Hom^{0}(\epsilon,\epsilon)-\dim B^{0}(\epsilon,\epsilon)-\ell
-(\chi_*(\Lambda)-tb(\Lambda))/2} =
\frac{| H^{-1} \Hom(\epsilon, \epsilon) | \cdot | H^{-3} \Hom(\epsilon, \epsilon) |
\cdots}{|H^{-2} \Hom(\epsilon,
\epsilon)| \cdot |H^{-4} \Hom(\epsilon,\epsilon)| \cdots}
\]
and so the left hand side is indeed
$\pi_{\ge 0} \Aug(\Lambda; \F_q)^*$.

For $m>0$, there are two deficiencies with viewing the left hand side
in Proposition~\ref{prop:2mcard} as the homotopy cardinality: first,
there is no obvious relation to the notion of homotopy cardinality
from homotopy theory; and second, the expression involves ranks of chain
and boundary groups, rather than cohomology groups, which are more
natural from the viewpoint of proving invariance directly. We do not
address the first problem here, but propose a conjectural way to
resolve the second.
There is a fair amount of experimental evidence for the following.

\begin{conjecture}
Let $\epsilon :\thinspace \cA(\Lambda) \to \F_q$ be a $2m$-graded
augmentation of a Legendrian link $\Lambda$ with $\ell$ components, for $m\geq 0$.
\label{conj:2mgraded}
Then
\[
2\dim\Hom^{0}(\epsilon,\epsilon)-2\dim B^{0}(\epsilon,\epsilon)-\ell
-\chi_*(\Lambda) =
2\dim H^0\Hom(\epsilon,\epsilon)-\dim H^1\Hom(\epsilon,\epsilon).
\]
\end{conjecture}

Conjecture~\ref{conj:2mgraded}, combined with
Proposition~\ref{prop:2mcard}, would have the following consequence.

\begin{corollary}
For any $m\geq 0$,\label{cor:2mgraded}
\[
\sum_{[\epsilon]} \frac{1}{\lvert\Aut(\epsilon)\rvert} \frac{|H^0\Hom(\epsilon,\epsilon)|}{|H^1\Hom(\epsilon,\epsilon)|^{1/2}} q^{(tb(\Lambda)-\ell)/2} = q^{tb(\Lambda)/2} \sum_{R \in \cR} (q^{1/2} - q^{-1/2})^{-\chi(R)}.
\]
\end{corollary}

\noindent
It would then be reasonable to define the left hand side of
Corollary~\ref{cor:2mgraded} as the homotopy cardinality of the
$2m$-graded augmentation category. 

\begin{remark}
In the case $m=0$, the left hand
side agrees with the definition of $\pi_{\ge 0} \Aug(\Lambda;
\F_q)^*$ from the Introduction, as can be seen directly by Sabloff
duality (see the proof of Proposition~\ref{prop:0graded} below).
For $m=1$, since $tb(\Lambda) = \dim H^1\Hom(\e,\e) - \dim H^0\Hom(\e,\e)$, the left hand side of Corollary~\ref{cor:2mgraded} is
$\sum_{[\epsilon]} \frac{1}{\lvert\Aut(\epsilon)\rvert} |H^0\Hom(\e,\e)|^{1/2} q^{-\ell/2}$.
\end{remark}

We conclude this section by establishing
Conjecture~\ref{conj:2mgraded} in a case that generalizes $m=0$.
Suppose that $m$ is arbitrary but the augmentation $\epsilon$ is not just
$\Z/(2m)$-graded but in fact $\Z$-graded, i.e., $\epsilon(a) = 0$
unless $|a|=0$. Then the space $\Hom(\epsilon,\epsilon)$ has a $\Z$
grading that descends to a $\Z/(2m)$ grading.

\begin{proposition}
Suppose that $\epsilon$ is $\Z$-graded and thus $\Z/(2m)$-graded for
any $m\geq 0$. Then
Conjecture~\ref{conj:2mgraded} holds for any $m\geq 0$.
\label{prop:0graded}
\end{proposition}

\noindent
To clarify, the gradings in
Conjecture~\ref{conj:2mgraded} are taken mod $2m$. Thus when
$\epsilon$ is $\Z$-graded,
the conclusion of Conjecture~\ref{conj:2mgraded} can be written:
\[
2\sum_k (\dim\Hom^{2mk}(\epsilon,\epsilon)-\dim B^{2mk}(\epsilon,\epsilon))-\ell
-\chi_*(\Lambda) =
\sum_k (2\dim H^{2mk}\Hom(\epsilon,\epsilon)-\dim H^{2mk+1}\Hom(\epsilon,\epsilon)),
\]
where all gradings are now taken over $\Z$.

\begin{proof}[Proof of Proposition~\ref{prop:0graded}]
For convenience, we write
$\Hom^i(\epsilon,\epsilon)$ as $\Hom^i$, $H^i\Hom(\epsilon,\epsilon)$
as $H^i$, and $B^i(\epsilon,\epsilon)$ as
$B^i$.
Note that $\Hom^i$ is generated by the Reeb chords of $\Lambda$ of degree
$i-1$, along with $\ell$ generators in degree $0$ and $\ell$
generators in degree $1$. Thus
\begin{align*}
2\sum_k\dim\Hom^{2mk}-2\ell-\chi_*(\Lambda) &=
\sum_k (2k+1)(\dim\Hom^{2mk+2m}-\dim\Hom^{2mk}) \\
&\qquad - \sum_k
(2k+1)\sum_{i=0}^{2m-1}(-1)^i\dim\Hom^{2mk+i+1} \\
&= \sum_k (2k+1) \sum_{i=0}^{2m-1} (-1)^i \dim\Hom^{2mk+i} \\
&= \sum_k (2k+1)\sum_{i=0}^{2m-1} (-1)^i\dim H^{2mk+i} + 2\sum_k
\dim B^{2mk},
\end{align*}
where the final equality comes from taking the alternating sum of
$\dim\Hom^i = \dim B^i+\dim H^i+\dim B^{i+1}$ from $i=0$ to $i=2m-1$.

Thus it suffices to show that
\[
\sum_k (2k+1)\sum_{i=0}^{2m-1} (-1)^i\dim H^{2mk+i} =
\sum_k (2\dim H^{2mk}-\dim H^{2mk+1})-\ell.
\]
But this follows, after a bit of algebraic manipulation, from Sabloff
duality (\cite{EESab}, cf.\
\cite[section 5.1.2]{NRSSZ}), which implies that $\dim H^i = \dim H^{2-i}$ unless
$i=0$ or $i=2$, and $\dim H^0 = \dim H^2 + \ell$.
\end{proof}

\bibliographystyle{alpha}
\bibliography{counting}

\newcommand{\etalchar}[1]{$^{#1}$}
\def\cprime{$'$}
\begin{thebibliography}{NRS{\etalchar{+}}15}

\bibitem[Bar94]{Bar}
S.~A. Barannikov.
\newblock The framed {M}orse complex and its invariants.
\newblock {\em Adv. Soviet Math.}, 21:93--115, 1994.

\bibitem[BC14]{BC}
Fr{\'e}d{\'e}ric Bourgeois and Baptiste Chantraine.
\newblock Bilinearized {L}egendrian contact homology and the augmentation
  category.
\newblock {\em J. Symplectic Geom.}, 12(3):553--583, 2014.

\bibitem[BD01]{BD}
John~C. Baez and James Dolan.
\newblock From finite sets to {F}eynman diagrams.
\newblock In {\em Mathematics unlimited---2001 and beyond}, pages 29--50.
  Springer, Berlin, 2001.

\bibitem[Cha10]{Cha}
Baptiste Chantraine.
\newblock Lagrangian concordance of {L}egendrian knots.
\newblock {\em Algebr. Geom. Topol.}, 10(1):63--85, 2010.

\bibitem[Cha13]{Cha-symmetric}
Baptiste Chantraine.
\newblock Lagrangian concordance is not a symmetric relation.
\newblock arXiv:1301.3767, 2013.

\bibitem[Che02]{Che}
Yuri Chekanov.
\newblock Differential algebra of {L}egendrian links.
\newblock {\em Invent. Math.}, 150(3):441--483, 2002.

\bibitem[CNS]{CNS}
Christopher~R. Cornwell, Lenhard Ng, and Steven Sivek.
\newblock Obstructions to {L}agrangian concordance.
\newblock {\em Algebr. Geom. Topol.}
\newblock To appear.

\bibitem[CP05]{CP}
Yu.~V. Chekanov and P.~E. Pushkar{\cprime}.
\newblock Combinatorics of fronts of {L}egendrian links, and {A}rnol\cprime d's
  4-conjectures.
\newblock {\em Uspekhi Mat. Nauk}, 60(1(361)):99--154, 2005.
\newblock Translation in Russian Math. Surveys, 60(1):95--149, 2005.

\bibitem[EES05]{EESorient}
Tobias Ekholm, John Etnyre, and Michael Sullivan.
\newblock Orientations in {L}egendrian contact homology and exact {L}agrangian
  immersions.
\newblock {\em Internat. J. Math.}, 16(5):453--532, 2005.

\bibitem[EES09]{EESab}
Tobias Ekholm, John~B. Etnyre, and Joshua~M. Sabloff.
\newblock A duality exact sequence for {L}egendrian contact homology.
\newblock {\em Duke Math. J.}, 150(1):1--75, 2009.

\bibitem[EGH00]{EGH}
Y.~Eliashberg, A.~Givental, and H.~Hofer.
\newblock Introduction to symplectic field theory.
\newblock {\em Geom. Funct. Anal.}, (Special Volume, Part II):560--673, 2000.
\newblock GAFA 2000 (Tel Aviv, 1999).

\bibitem[EHK]{EHK}
Tobias Ekholm, Ko~Honda, and Tam\'{a}s K\'{a}lm\'{a}n.
\newblock Legendrian knots and exact {L}agrangian cobordisms.
\newblock {\em J. Eur. Math. Soc. (JEMS)}.
\newblock To appear.

\bibitem[Ekh07]{Ekholmtrees}
Tobias Ekholm.
\newblock Morse flow trees and {L}egendrian contact homology in 1-jet spaces.
\newblock {\em Geom. Topol.}, 11:1083--1224, 2007.

\bibitem[Eli98]{Eli}
Yakov Eliashberg.
\newblock Invariants in contact topology.
\newblock In {\em Proceedings of the {I}nternational {C}ongress of
  {M}athematicians, {V}ol. {II} ({B}erlin, 1998)}, number Extra Vol. II, pages
  327--338, 1998.

\bibitem[ENS02]{ENS}
John~B. Etnyre, Lenhard~L. Ng, and Joshua~M. Sabloff.
\newblock Invariants of {L}egendrian knots and coherent orientations.
\newblock {\em J. Symplectic Geom.}, 1(2):321--367, 2002.

\bibitem[EP96]{EP}
Y.~Eliashberg and L.~Polterovich.
\newblock Local {L}agrangian {$2$}-knots are trivial.
\newblock {\em Ann. of Math. (2)}, 144(1):61--76, 1996.

\bibitem[FI04]{FI}
Dmitry Fuchs and Tigran Ishkhanov.
\newblock Invariants of {L}egendrian knots and decompositions of front
  diagrams.
\newblock {\em Mosc. Math. J.}, 4(3):707--717, 783, 2004.

\bibitem[Fuc03]{Fuc}
Dmitry Fuchs.
\newblock Chekanov-{E}liashberg invariant of {L}egendrian knots: existence of
  augmentations.
\newblock {\em J. Geom. Phys.}, 47(1):43--65, 2003.

\bibitem[HR15]{HR}
Michael~B. Henry and Dan Rutherford.
\newblock Ruling polynomials and augmentations over finite fields.
\newblock {\em J. Topol.}, 8(1):1--37, 2015.

\bibitem[MP07]{JP}
Jo{\~a}o~Faria Martins and Timothy Porter.
\newblock On {Y}etter's invariant and an extension of the {D}ijkgraaf-{W}itten
  invariant to categorical groups.
\newblock {\em Theory Appl. Categ.}, 18:No. 4, 118--150, 2007.

\bibitem[MS05]{MS}
Paul Melvin and Sumana Shrestha.
\newblock The nonuniqueness of {C}hekanov polynomials of {L}egendrian knots.
\newblock {\em Geom. Topol.}, 9:1221--1252 (electronic), 2005.

\bibitem[NRS{\etalchar{+}}15]{NRSSZ}
Lenhard Ng, Dan Rutherford, Vivek Shende, Steven Sivek, and Eric Zaslow.
\newblock Augmentations are sheaves.
\newblock arXiv:1502.04939, 2015.

\bibitem[NS06]{NS}
Lenhard~L. Ng and Joshua~M. Sabloff.
\newblock The correspondence between augmentations and rulings for {L}egendrian
  knots.
\newblock {\em Pacific J. Math.}, 224(1):141--150, 2006.

\bibitem[NZ09]{NZ}
David Nadler and Eric Zaslow.
\newblock Constructible sheaves and the {F}ukaya category.
\newblock {\em J. Amer. Math. Soc.}, 22(1):233--286, 2009.

\bibitem[Rut06]{Rut}
Dan Rutherford.
\newblock Thurston-{B}ennequin number, {K}auffman polynomial, and ruling
  invariants of a {L}egendrian link: the {F}uchs conjecture and beyond.
\newblock {\em Int. Math. Res. Not.}, pages Art. ID 78591, 15, 2006.

\bibitem[Sab05]{Sab}
Joshua~M. Sabloff.
\newblock Augmentations and rulings of {L}egendrian knots.
\newblock {\em Int. Math. Res. Not.}, (19):1157--1180, 2005.

\bibitem[She15]{She}
Vivek Shende.
\newblock Generating families and constructible sheaves.
\newblock arXiv:1504.01336, 2015.

\bibitem[STZ14]{STZ}
Vivek Shende, David Treumann, and Eric Zaslow.
\newblock Legendrian knots and constructible sheaves.
\newblock arXiv:1402.0490, 2014.

\end{thebibliography}


\begin{thebibliography}{10}

\bibitem{BD} John C. Baez and James Dolan, {\em From finite sets to Feynman diagrams},
arXiv:math/0004133.

\bibitem{Bar}  S. A. Barannikov, {\em The framed Morse complex and its invariants},
  Adv. Soviet Math. 21 (1994), 93-115.

\bibitem{Cha} Baptiste Chantraine, {\em Lagrangian concordance of Legendrian knots},
Algebr. Geom. Topol. 10 (2010), no. 1, 63--85.

\bibitem{Che} Yu. Chekanov, {\em Differential algebra of Legendrian links},
Invent. Math. 150 (2002), no. 3, 441--483.

\bibitem{CP}
Yu. Chekanov and P. Pushkar, {\em
Combinatorics of Legendrian links and the Arnol'd 4-conjectures},
Uspekhi Mat. Nauk 60, no. 1, 99--154,
translated in Russian Math. Surveys 60, no. 1, 95--149.

\bibitem{CNS}
Christopher R. Cornwell, Lenhard Ng, and Steven Sivek,
{\em Obstructions to Lagrangian concordance},
arXiv:1411.1364.

\bibitem{Eli} Ya. Eliashberg, {\em Invariants in contact topology},
Doc. Math. 1998, Extra Vol. II, 327--338 (electronic).

\bibitem{EGH} Ya. Eliashberg, A. Givental, H. Hofer, {\em Introduction to Symplectic field theory},
arXiv:math/0010059.

\bibitem{EHK} Tobias Ekholm, Ko Honda, and Tam\'as K\'alm\'an, {\em Legendrian knots and exact Lagrangian cobordisms}, arXiv:1212.1519.

\bibitem{FHT}
Y. 
F\'elix,
S. 
Halperin, and
J.-C. 
Thomas, \underline{Rational homotopy theory},
(Springer-Verlag, 2001).

\bibitem{Fuc} D. Fuchs, {\em
Chekanov--Eliashberg invariant of Legendrian knots: existence of augmentations},
J. Geom. Phys. 47 (2003), no. 1, 43--65.

\bibitem{FI} D. Fuchs and T. Ishkhanov, {\em
Invariants of Legendrian knots and decompositions of front diagrams},
Mosc. Math. J. 4 (2004), no. 3, 707--717.

\bibitem{HR} M. B. Henry and D. Rutherford,
{\em Ruling polynomials and augmentations over finite fields}, arXiv:1308.4662.

\bibitem{JP} Jo\~ao Faria Martins and Timothy Porter, {\em On Yetter's invariant
and an extension of the Dijkgraaf-Witten invariant to categorical groups},
Theory and Applications of Categories {\bf 18}.4 (2007), 118--150.

\bibitem{MS} Paul Melvin and Sumana Shrestha,
\em{The nonuniqueness of Chekanov polynomials of Legendrian knots},
Geom. Topol. 9 (2005), 1221--1252 (electronic).

\bibitem{Nad} D. Nadler, {\em Microlocal branes are constructible sheaves}, arxiv:math/0612399.

\bibitem{NZ} D. Nadler and E. Zaslow, {\em Constructible Sheaves and the Fukaya Category},
arxiv:math/0604379.

\bibitem{NRSSZ} L. Ng, D. Rutherford, V. Shende, S. Sivek, E. Zaslow,
{\em Augmentations are sheaves},
arxiv: 1502.04939.

\bibitem{NS} L. Ng and J. Sabloff, {\em The correspondence between
augmentations and rulings for Legendrian knots}, arxiv:math/0503168.

\bibitem{Rut} D. Rutherford, {\em
The Bennequin number, Kauffman polynomial, and ruling invariants of a Legendrian link: the Fuchs conjecture and beyond}, arxiv:math/0511097.

\bibitem{Sab} J. Sabloff, {\em Augmentations and rulings of Legendrian
  knots}, Int. Math. Res. Not. 2005, no. 19, 1157--1180.


\bibitem{She} V. Shende, {\em Generating families and constructible sheaves},
arxiv:1504.01336.

\bibitem{STZ} V. Shende, D. Treumann, E. Zaslow, {\em Legendrian knots and constructible sheaves},
arxiv:1402.0490.

\end{thebibliography}

\newcommand{\longcomment}[2]{#2}
\longcomment{

} 

\end{document}